\let\ds=\displaystyle

\def\be{\begin{eqnarray}}
\def\en{\end{eqnarray}}
\def\nn{\nonumber}
\def\RR{\mathbb R}

\documentclass[11pt]{article}
\usepackage{graphicx}
\usepackage{amsmath,amsfonts,amsxtra,amssymb,latexsym, amscd,amsthm}

\let\ds=\displaystyle
\newtheorem{theorem}{Theorem}[section]
\newtheorem{lemma}[theorem]{Lemma}

\begin{document}

\title{Large time behavior of differential equations with drifted periodic coefficients modeling Carbon storage in soil}

\author{Stephane Cordier $^{(1)}$, Le Xuan Truong$^{(2)}$, \\
            Nguyen Thanh Long $^{(3)}$, Alain Pham Ngoc Dinh
            $^{(1)}$ }
\date{}
\maketitle {\small
\begin{list}{}{
\setlength{\rightmargin}{.4in} \setlength{\leftmargin}{.4in}
\setlength{\parsep}{.2pt}}
\item[$^{(1)}$] MAPMO, UMR CNRS 6628, F\'ed\'eration Denis Poisson, b\^{a}t. Math\'{e}matiques, University of Orl\'{e}ans, BP 6759, 45067 Orl\'{e}eans Cedex 2, France.\\
\hspace{0.2cm}{E-mail: stephane.cordier@univ-orleans.fr, alain.pham@univ-orleans.fr}
\item[$^{(2)}$]Department of Mathematics, University of Economics of HoChiMinh City, 59C Nguyen Dinh Chieu Str., Dist. 3, HoChiMinh City, Vietnam.
 \\
\hspace{0.2cm}{E-mail: lxuantruong@gmail.com }
\item[$^{(3)}$] Department of Mathematics and Computer Science, University of Natural Science, Vietnam National University HoChiMinh City,
227 Nguyen Van Cu Str., Dist.5, HoChiMinh City, Vietnam.\
\hspace{0.2cm}{E-mail:longnt@hcmc.netnam.vn, longnt2@gmail.com }
\end{list}
}
\hrule
{\small {\center{\bf Abstract}}

\ \par This paper is concerned with the linear ODE in the form
$y'(t)=\lambda\rho(t)y(t)+b(t)$, $\lambda <0$  which represents
a simplified storage model
of the carbon in the soil. In the first
part, we show
that, for a periodic function $\rho(t)$, a linear drift in the
coefficient $b(t)$ involves a linear drift for the solution of this
ODE. In the second
part, we extend the previous results to a classical heat
non-homogeneous equation. The connection with an analytic semi-group
associated to the ODE equation is considered in the third part.
Numerical examples are given.

\ \par
 \vspace{0.2cm}

\vspace{0.3cm} \hrule \ \par \ \par

\section{Introduction}

A lot of phenomena of evolution are described using ordinary differential
equations ODE or systems in which the coefficients and/or the source
terms are periodic. Let us mention some applications in physics
(e.g. the harmonic oscillator , the resonance phenomena due to
oscillatory source terms), electricity (let us mention the famous
RLC circuit with an oscillatory generator), in biology (circadian
cycle), in agricultural studies (due to seasonal effects).\ \par

The main question which is addressed here is when or under which
conditions a slow perturbation of the coefficient in the ODE will
induce a similar behavior on the solutions, in large time. More
precisely, we are looking for conditions to ensure that a linear
drift in the (periodic) coefficients of the ODE will lead to a
linear deviation (and thus unbounded) of the solutions or, on the
contrary, what kind of perturbation in the coefficients are
compatible with a stable (bounded, periodic in large time)
solutions.

Although these questions can be applied to large number of applications,
our original motivation concerned with the effect of climate change on
the seasonal variations in the organic carbon contained in the soil
as claimed at the end of the conclusion of Martin et al.  [13].\par

Since the readers may be not familiar with this domain. Let
 us recall some basis about the issues of
  the soil organic carbon (SOC) modelling.\par

The spreading is one of opportunities for organic materials of human
origin (sludge of filtering treatment station and derivative
products) and agricultural (manure). The organic matter spread
contain significant amounts of organic carbon which, after
application, a fraction is permanently stored in the soil. ; The
remainder is returned to the atmosphere as CO2. The spreading can be
accessed through the storage of organic carbon in the soil, helping
to reduce CO2 emissions (a major greenhouse gas effect) compared for
example to the incineration of organic matter that returns all
carbon into the atmosphere. The optimization of spreading of organic
materials is important in reducing emissions of greenhouse gas
effect.\ \par

 The dynamics of carbon in the soil, which determines
the amount of organic carbon stored in soil and returned in the form
of CO2, depends on soil type, agricultural practices, climate and
quantities spread. The soil organic carbon (SOC) plays an important
role in several environmental and land management issues. One of the
most important issues is the role that SOC plays as part of the
terrestrial carbon and might play as a regulation of the atmospheric
CO2. Many factors are likely, in a near future, to modify the SOC
content, including changes in agricultural practices [3,18,2] and global
climate changes [9,5,8,10,11]. \par

Understanding SOC as a function of soil characteristics, agricultural management
and climatic conditions is therefore crucial and many models have
been developed in this perspective. These models are used in a
variety of ways and after for long term studies
[6,15,16,17]. The behavior of the SOC system, over a long term and assuming that
the environment of the system (inputs of organic carbon, climatic
conditions) is stable, is reported to tend towards steady or periodic state. \par

Some of he SOC models have been formulated mathematically [14,4,1,12,13].
 Here we consider the RothC model  [7,13] which consists in
splitting the soil carbon into four active compartments. Under a
continuous form it can be written as
\begin{equation}
\frac{dC(t)}{dt}=\rho(t)AC(t)+B(t),  \tag{1.1}
\end{equation}
 where C (t) is a vector with 4
components, each corresponding to a compartment storage of carbon in
the soil: DPM (decomposable plant material), RPM ( resistant plant
material), BIO (microbial community) and HUM ( humus). These
components indicate the amount of carbon stored at the moment t.\
\par In (1.1) \ \par
i) $\rho(t)$ is a positive function indicating the speed of mineralization of
soil, which results in CO2 emissions. \ \par ii) $A$ is a matrix
that can represent the percentage of clay in the soil and speeds of
mineralization in each compartment of C in the soil.\ \par
 iii) $B(t)$ indicates the amount of carbon brought in the soil  (amount
spread per unit time).\ \par

Let us recall that the initial goal of this study  was to
understand how long term evolution on
climatic data imply variation in large time on the solution of
ODE with periodic coefficient and/or source
terms with a (linear) drift.
The mathematical tools involved in this paper are rather classical
and simple but there are, up to our knowledge, very few literature
on the subject. \par

This article is organized as follows: \ \par
In section 2 we consider the asymptotic behavior oh the solution of the ODE associated to (1.1)
and we derive the main property i.e. there exits a unique solution $y_{\infty}(t)$ of
$y^{\prime }\left( t\right) =\lambda \rho \left( t\right) y\left( t\right)
            +b\left( t\right) ,\text{ }0\leq t<+\infty ,$ such that  $y_{\infty}(t+T)=y_{\infty}(t)+\gamma(t)$
where $\gamma(t)$ is a periodic function of period $T$. In section 3 we consider a classical heat non-homogeneous
 equation whose the rhs $ f(x,t)$ satisfies the assumption derived from the 2nd section. So
the properties concerning this equation can be deduced  from the
results of the ODE ( theorem 3.2 and theorem 3.3). In the 3rd part (section 4)
we connect the PDE and ODE equations with an analytic semi-group to
extend the previous results given by the PDE and ODE equations (
theorem 4.2). These results of sections 2 and 3 are illustrated by numerical tests in
section 5.

\section{Asymptotic behavior of the solution of the ODE}

 In this paper, we consider the linear differential equation
    \begin{equation}
            y^{\prime }\left( t\right) =\lambda \rho \left( t\right) y\left( t\right)
            +b\left( t\right) ,\text{ }0\leq t<+\infty , \tag{2.1}
    \end{equation}
where $\lambda < 0$, $\rho(t)$ and $b(t)$ are given functions satisfying the following conditions

\vspace{0.1cm}
\noindent $(A_{1})$ $\rho \left( t\right) $ is an T-periodic function, with $T>0$ fixed.

\vspace{2mm}
\noindent $(A_{2})$ there exist the T-periodic function, $\beta(t)$, such that
    \begin{equation}
         b(t+T)=b(t)+\beta(t), \quad \forall t \in \left[0, \infty\right). \tag{2.2}
    \end{equation}

\noindent The general solution of (2.1) has the form
    \begin{equation}
        y(t)=e^{a(t)}\left\{C_1+\int_0^t{b(s)e^{-a(s)}ds}\right\}, \tag{2.3}
    \end{equation}
where $C_1$ is a constant and
    \begin{equation}
        a(t)\stackrel{def}{=}\lambda\int_0^t{\rho(s)ds}. \tag{2.4}
    \end{equation}
In this section, we prove that there exists a unique solution
$y_{\infty}(t)$ of (2.1) satisfying
    \begin{equation}
        y_{\infty}(t+T)=y_{\infty}(t)+\gamma(t), \forall t\in [0, \infty), \tag{2.5}
    \end{equation}
where $\gamma(t)$ is an T-periodic function. Let first remark that
\begin{lemma}
 Let $ b:R _{+}\longrightarrow R $.
 The following properties are equivalent :\\
(a) $ \exists  \beta  (t)\ $periodic with period $ T$ such that\
 $ b(t+T)=b(t)+\beta  (t),\ \forall  t\geq  0$\\
(b) $  \exists  \tilde b(t)$ periodic with period $ T$ such
that $ b(t)=\tilde b(t)+{{\ds t}\over{\ds T}}  \beta (t),\forall t\geq 0$
\end{lemma}
\begin{proof}
Let us first prove  (a)$\Longrightarrow  $(b).
Choose $ \tilde b(t)=b(t)-{{\ds t}\over{\ds T}}  \beta (t)$. Then $
\tilde b(t)$ is periodic with period $ T$ since \\
  $ \tilde b(t+T)=b(t+T)-{{ t+T}\over{\ds T}}
\beta  (t+T)=b(t)+\beta  (t)-({{\ds t}\over{\ds T}} +1)\beta  (t)= b(t)-{{\ds t}\over{\ds T}}  \beta  (t)=\tilde
b(t).$\\

Conversely  (b)$\Longrightarrow  $(a). Using
 $ b(t)=\tilde b(t)+{{\ds t}\over{\ds T}}  \beta  (t)$, we have
 $ b(t+T)=\tilde b(t+T)+{{\ds t+T}\over{\ds T}}  \beta
(t+T)=\tilde b(t)+({{\ds t}\over{\ds T}}  +1)\beta  (t)
        =\tilde b(t)+{{\ds t}\over{\ds T}}  \beta
(t)+\beta  (t)=b(t)+\beta  (t).$ This concludes the proof.
\end{proof}

Let us now prove the main result of this section.
First, we state the some lemmas

\begin{lemma}
    Let $(A_{1})$ hold. Then
    \begin{equation}
    a(t+nT)=a(t)+a(nT)=a(t)+na(T), \forall t \geq 0, n\in \mathbb{N}. \tag{2.6}
    \end{equation}
\end{lemma}
\begin{proof}
    From (2.4) we deduce that
    \begin{equation}
    a(t+nT)=\lambda\int_0^{nT}\rho(s)ds+\lambda\int_{nT}^{t+nT}\rho(s)ds
               =a(nT)+\lambda\int_{nT}^{t+nT}\rho(s)ds. \tag{2.7}
    \end{equation}
On the other hand, by the assumption $(A_{1})$, we have
    \begin{equation}
    \lambda\int_{nT}^{t+nT}\rho(s)ds=\lambda\int_0^t{\rho(s+nT)ds}=\lambda\int_0^t
    {\rho(s)ds}=a(t), \tag{2.8}
    \end{equation}
and
    \begin{equation}
    a(nT)=\lambda\sum_{k=0}^{n-1}\int_{kT}^{(k+1)T}\rho(s)ds
            =\lambda\sum_{k=0}^{n-1}\int_{0}^{T}\rho(s)ds=na(T). \tag{2.9}
    \end{equation}
Combining (2.7)-(2.9) we have (2.6).
\end{proof}

\begin{lemma}
    Let assumptions $(A_1)$, $(A_2)$ hold. For $n\in \mathbb{Z}_+$ and $t \in [0, \infty)$, we put
    \begin{equation}
        y_{n}(t)=y(t+nT)=e^{a(t+nT)}\left(y(0)+\int_{0}^{t+nT}{b(s)e^{-a(s)}ds}\right). \tag{2.10}
    \end{equation}
Then,
    \begin{equation}
    y_{n}(t)=y_{\infty}(t)+\delta_{n}(t)+n\gamma(t), \tag{2.11}
    \end{equation}
where
    \begin{alignat}{1}
    y_{\infty}(t)=&e^{a(t)}\left\{\frac{e^{a(T)}}{1-e^{a(T)}}\int_{0}^{T}{b(s)e^{-a(s)}ds}
     -\frac{e^{a(T)}}{\left(1-e^{a(T)}\right)^2}\int_{0}^{T}{\beta(s)e^{-a(s)}ds}
    +\int_{0}^{t}{b(s)e^{-a(s)}ds}\right\}, \tag{2.12}
    \end{alignat}

    \begin{alignat}{1}
    \delta_{n}(t)=&e^{a(t)}e^{na(T)}\left\{y(0)-\frac{e^{a(T)}}{1-e^{a(T)}}\int_{0}^{T}
    {b(s)e^{-a(s)}ds}
     +\frac{e^{a(T)}}{\left(1-e^{a(T)}\right)^2}\int_{0}^{T}{\beta(s)e^{-a(s)}ds}\right\},
    \tag{2.13}
    \end{alignat}
and
    \begin{equation}
    \gamma(t)=e^{a(t)}\left\{\frac{e^{a(T)}}{1-e^{a(T)}}\int_{0}^{T}
    {\beta(s)e^{-a(s)}ds} + \int_{0}^{t}{\beta(s)e^{-a(s)}ds}\right\}. \tag{2.14}
    \end{equation}
\end{lemma}

\begin{proof}
By the assumption $(A_{2})$, it follows from (2.10) and the lemma
2.2 that
    \begin{equation}
    y_{n}(t)=e^{a(t)}\left\{y_{n}(0)+\int_0^t{b(s)e^{-a(s)}ds}+n\int_0^t{\beta(s)e^{-a(s)}ds}
    \right\}. \notag 
    \end{equation}
On the other hand, we have
    \begin{alignat}{1}
    y_{n}\left( 0\right)&=y\left( nT\right) =e^{a\left( nT\right)
    }\left\{
    y\left( 0\right) +\int_{0}^{nT}b\left( s\right) e^{-a\left( s\right)}ds\right\}   \notag \\
    &=e^{na\left( T\right) }\left\{ y\left( 0\right)+\sum_{k=0}^{n-1}\int_{0}^{T}b\left( s+kT\right)
    e^{-a\left( s+kT\right) }ds\right\}  \notag \\
    &=e^{na\left( T\right) }\left\{ y\left( 0\right) +\int_{0}^{T}b\left(s\right) e^{-a\left( s\right)
    }ds\sum_{k=0}^{n-1}e^{-ka\left( T\right)}+\int_{0}^{T}\beta \left( s\right) e^{-a\left( s\right)
    }ds\sum_{k=0}^{n-1}ke^{-ka\left( T\right) }\right\} . \notag
    \end{alignat}
By using the following equalities
    \begin{equation}
    \sum_{k=0}^{n-1}e^{-ka\left( T\right) }=\frac{1-e^{-na\left( T\right) }}{1-e^{-a\left( T\right)
    }}, \notag
    \end{equation}
and
    \begin{equation}
    \sum_{k=0}^{n-1}ke^{-ka\left( T\right) }=\frac{e^{-a\left( T\right) }}{\left( e^{-a\left( T\right)
    }-1\right) ^{2}}-\frac{e^{-\left( n+1\right)a\left( T\right) }}{\left( e^{-a\left( T\right) }-1\right)
    ^{2}}+n\frac{e^{-na\left( T\right) }}{e^{-a\left( T\right) }-1}, \notag
\end{equation}
 thus we obtain
    \begin{alignat}{1}
    y_{n}\left( 0\right)&=\frac{e^{a\left( T\right) }}{1-e^{a\left( T\right) }}\int_{0}^{T}b\left(
    s\right) e^{-a\left( s\right) }ds-\frac{e^{a\left(T\right) }}{\left( 1-e^{a\left( T\right) }\right)
    ^{2}}\int_{0}^{T}\beta\left( s\right) e^{-a\left( s\right) }ds   \notag \\ 
    &+e^{na\left( T\right) }\left\{ y\left( 0\right) -\frac{e^{a\left( T\right)}}{1-e^{a\left( T\right)
    }}\int_{0}^{T}b\left( s\right) e^{-a\left( s\right)}ds+\frac{e^{a\left( T\right) }}{\left(
    1-e^{a\left( T\right) }\right) ^{2}}\int_{0}^{T}\beta \left( s\right) e^{-a\left(
    s\right)}ds\right\}.    \notag \\
    &+n\frac{e^{a\left( T\right) }}{1-e^{a\left( T\right) }}\int_{0}^{T}\beta\left( s\right)
    e^{-a\left( s\right) }ds  \notag
    \end{alignat}
Combining previous equalities, we obtain (2.11). The proof of Lemma is
complete.
\end{proof}

Now, we state the main theorem
\begin{theorem}
    Let $(A_{1})$, $(A_{2})$ hold. Then, there exists a unique solution $y_{\infty}(t)$ of (2.1) such that
    \begin{equation}
        y_{\infty}(t+T)=y_{\infty}(t)+\gamma(t), \forall t\in [0, \infty), \tag{2.15}
    \end{equation}
where $\gamma(t)$ is periodic function of period T, defined by
    \begin{equation}
        \gamma(t)= e^{a(t)}\left\{\frac{e^{a(T)}}{1-e^{a(T)}}\int_0^T{\beta(s)e^{-a(s)}ds}
        +\int_0^t{\beta(s)e^{-a(s)}ds}\right\}.\tag{2.16}
    \end{equation}
\end{theorem}

\begin{proof}
For $n\in \mathbb{Z}_+$ and $t \geq 0$, let us define
    \begin{equation}
    u_{n}(t)\stackrel{def}{=}y_{n}(t)-n\gamma(t). \tag{2.17}
    \end{equation}
It follows from (2.11)-(2.14) and (2.17) that
    \begin{equation}
    \lim_{n\rightarrow +\infty }u_{n}\left( t\right) =y_{\infty }\left( t\right),\text{ }\forall t\in \left[
    0,+\infty \right). \tag{2.18}
    \end{equation}
It is clear that $y_{\infty}(t)$ is a solution of equation (2.1)
satisfies the value at $t=0$
    \begin{equation}
    y_{\infty}\left( 0\right) =\frac{e^{a\left( T\right) }}{1-e^{a\left( T\right) }}\int_{0}^{T}b\left( s\right)
    e^{-a\left( s\right) }ds-\frac{e^{a\left(T\right) }}{\left( 1-e^{a\left( T\right)
    }\right)^{2}}\int_{0}^{T}\beta\left( s\right) e^{-a\left( s\right) }ds\equiv L\left( T\right).
    \tag{2.19}
    \end{equation}
By (2.18), we have
    \begin{alignat}{1}
    y_{\infty }\left( t+T\right) &=\lim_{n\rightarrow +\infty }u_{n}\left(t+T\right)
    =\lim_{n\rightarrow +\infty }\left\{ y_{n}\left( t+T\right)-n\gamma \left( t+T\right) \right\}.
    \tag{2.20}\\
    &=\lim_{n\rightarrow +\infty }\left\{ y_{n+1}\left( t\right) -\left(n+1\right) \gamma \left(
    t\right) +n\left[ \gamma \left( t\right) -\gamma\left( t+T\right) \right] +\gamma \left( t\right)
    \right\}  \notag
    \end{alignat}
On the other hand, by the periodicity of $\beta(t)$, we get
    \begin{alignat}{1}
    \gamma \left( t+T\right) &=e^{a\left( t+T\right) }\left\{ \frac{e^{a\left(T\right) }}
    {1-e^{a\left( T\right) }}\int_{0}^{T}\beta \left( s\right)e^{-a\left( s\right) } ds
    +\int_{0}^{t+T}\beta \left( s\right) e^{-a\left(s\right) }ds\right\}  \tag{2.21} \\
    &=e^{a\left( t\right) }\left\{ \frac{e^{a\left( T\right) }}{1-e^{a\left(T\right) }}\int_{0}^{T}\beta
    \left( s\right) e^{-a\left( s\right)}ds+\int_{T}^{t+T}\beta \left( s\right) e^{a\left( T\right)
    -a\left(s\right) }ds\right\} \notag \\
    &=e^{a\left( t\right) }\left\{ \frac{e^{a\left( T\right) }}{1-e^{a\left(T\right) }}\int_{0}^{T}\beta
    \left( s\right) e^{-a\left( s\right)}ds+\int_{0}^{t}\beta \left( s\right) e^{-a\left( s\right)
    }ds\right\}=\gamma \left( t\right).  \notag
    \end{alignat}
Combining (2.20) and (2.21), we obtain
    \begin{equation}
    y_{\infty }\left( t+T\right) =\lim_{n\rightarrow +\infty }u_{n+1}\left(t\right) +\gamma \left(
    t\right) =y_{\infty }\left( t\right) +\gamma \left(t\right) .\tag{2.22}
    \end{equation}
\underline {Uniqueness} \ \par
 Now, let $\widetilde{y}\left(
t\right)$ be the solution of (2.1) corresponding to the initial
value $\widetilde{y}\left(0\right)=A$ and
    \begin{equation}
    \widetilde{y}\left( t+T\right) =\widetilde{y}\left( t\right) +\widetilde{\gamma }\left( t\right),
    \tag{2.23}
    \end{equation}
where $\widetilde{\gamma }\left( t\right)$ is an T-periodic function. Then
$y^{\ast }\left( t\right)=y_{\infty }\left( t\right) -\widetilde{y}\left( t\right)$ satisfy
    \begin{equation}
    \left\{ \begin{array}{l}
    y^{\prime }\left( t\right) =\lambda \rho \left( t\right) y\left( t\right) ,\text{ \ }0<t<+\infty,\\
    y\left( 0\right) =L\left( T\right) -A,
    \end{array}\right. \tag{2.24}
    \end{equation}
and
    \begin{equation}
    y^{\ast }\left( t+T\right) =y^{\ast }\left( t\right) +\gamma ^{\ast }\left(t\right) ,\text{ \
    }\gamma ^{\ast }\left( t+T\right) =\gamma ^{\ast }\left(t\right) ,\text{ }\forall t\geq 0.\text{ }
    \tag{2.25}
    \end{equation}
It follows from (2.24) that
    \begin{equation}
    y^{\ast }\left( t\right) =\left( L\left( T\right) -A\right) e^{a\left(t\right) },\forall t\geq 0.
    \tag{2.26}
    \end{equation}
From (2.25) and (2.26) we deduce that
    \begin{equation}
    \gamma ^{\ast }\left( t\right) =-\left( L\left( T\right) -A\right) \left(1-e^{a\left( T\right) }\right)
    e^{a\left( t\right) },\forall t\geq 0. \tag{2.27}
    \end{equation}
Combining (2.25), (2.27) we get $A=L(T)$. By the uniqueness of
Cauchy problem, the proof of theorem 2.8 is complete.
\end{proof}
$\underline{\rm Remark}$ \ \ If we consider the equation (2.1) where
the assumptions $(A_{1})$ and $(A_{2})$ are replaced by

\vspace{0.1cm} \noindent $(A_{1}^{\prime })$ $b \left( t\right) $ is
an T-periodic function, with $T>0$ fixed

\vspace{2mm} \noindent $(A_{2}^{\prime })$ there exist a
T-periodic function, $\alpha(t)$, such that
    \begin{equation}
         \rho(t+T)=\rho(t)+\alpha(t), \quad \forall t \in \left[0, \infty\right). \tag{2.28}
\end{equation}
In that case it is clear that there does not exist a solution which
has the same property as the function $\rho(t)$, for instance if we
consider the example with $b(t)=0$. Here the solution of (2.1) tends
to $0$ as $t \rightarrow +\infty$.

We end this section with an example. Consider the following Cauchy
problem (2.1) with the choice
\begin{equation}
\lambda  =-1,\; y_{0}=1,\; \rho  (t)={\rm sin}^{2}t,\; b(t)=t
\tag{2.29}
  \end{equation}
In fig.1 we have put the graphs of the functions $ y_{n}(t)$ and $
y_{n}(t+\pi )$ with $ n=5$ and here we also note the drift property
for the solution of (2.1) taking initial value $ y_{0}=1$ at $ t=0.$

\begin{center}
\includegraphics [width=10cm,height=10cm,
 keepaspectratio=true]{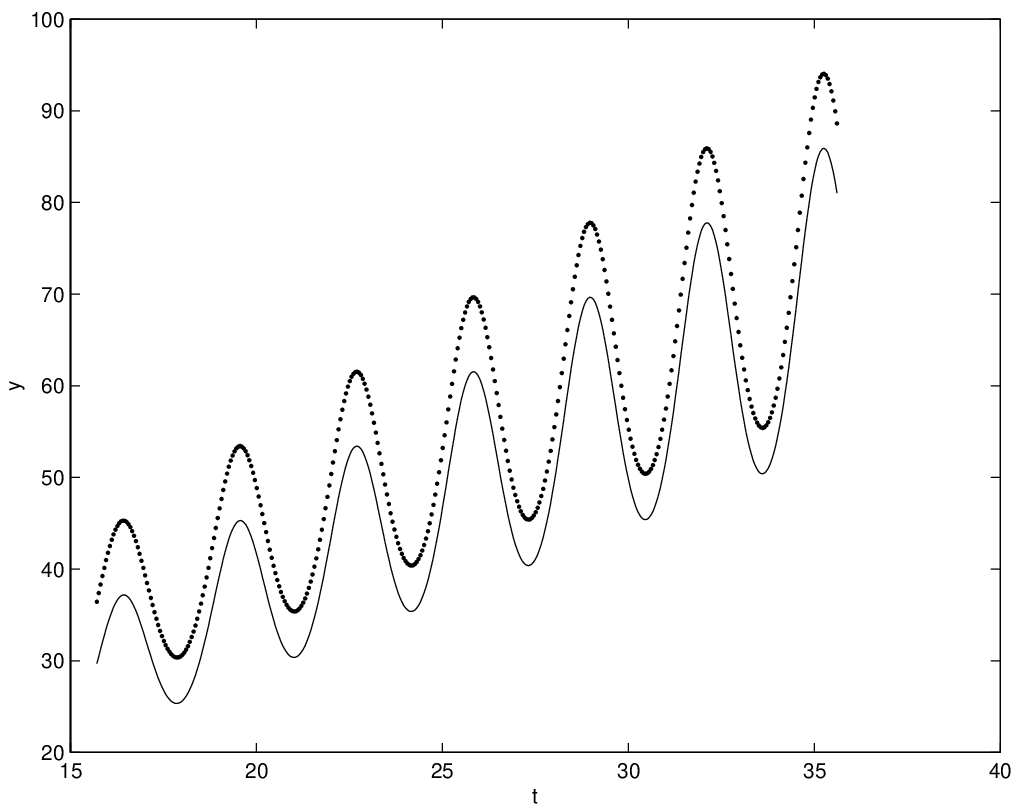}

Fig1: Drift property in asymptotic behavior
\end{center}

\section{Classical heat equation}

 \rm  Let $\Omega =(0,1)$ and $Q_T=\Omega \times (0,T)$, for
$T>0.$ In what follows we will denote
 \be
 \langle u,v\rangle =\int_0^1 u(x)v(x)dx, \|v \|=\sqrt{\langle v,v\rangle}.\nn \en
We also denote $u(x,t), u_t(x,t), u_x(x,t)$ by $u(t), u'(t), u_x(t)$
respectively. In this paper, first we consider the classical heat
equation
\begin{alignat}{1}
u_t-\rho(t)u_{xx}=f(x,t), ~~~(x,t)\in Q_T \tag{3.1}\\
 u(0,t)=u(1,t)=0, ~~~\tag{3.2}\\
 u(x,0)=u_0(x),~~~~\tag{3.3}.\end{alignat}
 In the next three theorems
$\rho(t)$ will be taken equal to $1$. We have the well known result for linear parabolic equation in the following theorem \ \par{\bf Theorem 3.1} \it
Let $T>0$ and assume that $u_0 \in L^{2}(\Omega), f\in L^{2}(Q_T)$.
Then, the problem (3.1)-(3.3) has a unique solution u satisfying \be
u\in L^{2}(0,T; H_{0}^{1}(\Omega))\cap C(0,T;L^{2}(\Omega)),\,
u_t\in L^{2}(0,T; H^{-1}(\Omega)). \nn \en

Assumption A}
\par Always with a rhs $f$ belonging $L^2_{Loc}(Q_\infty), Q_\infty=\Omega \times\RR_+$ we assume in addition
that $f$ satisfies the following hypothesis: \\
(\textit{A}) there exists the $T_1$ -periodic function $\beta
(x,t),$ such that
\begin{alignat}{1}  f(x,t+T_1)=f(x,t)+\beta (x,t), \quad\forall (x,t)\in Q_\infty \notag \\
 \beta(x,t+T_1)=\beta(x,t), \quad\forall (x,t)\in
Q_\infty. \tag{3.4}\end{alignat} It follows from (3.4) that the
functions $f_j(t)$ satisfies a drift property i.e.
\begin{equation} f_j(t+T_1)=f_j(t)+\beta_j(t)\tag{3.5} \end{equation}
where $\beta_j(t)=\int_0^1\beta(x,t)w_j(x) dx $ is a $T_1$
periodic-function. \\{\bf Theorem 3.2} \it Let $T>0$ and assume that
$f$ belonging $L^2_{Loc}(Q_\infty)$ and satisfies assumption $(A)$.
Then, the problem (3.1)-(3.2) has a unique solution $\tilde u$
satisfying \be \tilde u\in L^{2}(0,T; H_{0}^{1}(\Omega))\cap
C(0,T;L^{2}(\Omega)),\, \tilde u_t\in L^{2}(0,T; H^{-1}(\Omega)).
\nn \en\\ and which is $T_1$
periodic wrt the norm $L^2(\Omega).$ \\
\noindent Proof: \rm If we associate the ODE with the
drift-property (3.5) and theorem 3.1, we can then apply the previous results (Lemma
2.3 and theorem 2.4) in which we have proved that there exists a
unique solution called $g_j^\infty (t)$ and which has also the same
drift property $g_j^\infty (t+T_1)=g_j^\infty (t)+\gamma_j(t)$,
$\gamma_j(t)$ being $T_1$ -periodic.\\
Consider now the function \begin{equation} \tilde u(x,t)=\ds
\sum_{j=1}^\infty g_j^\infty(t)w_j(x)\tag{3.6}\end{equation} To
prove that the function $\tilde u$ is a solution of the problem
(3.1)-(3.3), then from theorem 3.1 it is sufficient to prove that
the initial condition $\tilde u(x,0)$ is in $L^2(\Omega).$ We have
\begin{equation}
g_j^\infty(0)=\frac{\textrm{e}^{-\lambda_jT_1}}{1-\textrm{e}^{-\lambda_jT_1}}\int_0^{T_1}
f_j(s)\textrm{e}^{\lambda_js}ds-
\frac{\textrm{e}^{-\lambda_jT_1}}{(1-\textrm{e}^{-\lambda_jT_1})^2}\int_0^{T_1}
\beta_j(s)\textrm{e}^{\lambda_js}ds. \tag{3.7}\end{equation} By
virtue of (3.7) and using the Cauchy-Schwarz inequality we get
after some calculations \be
|g_j^\infty(0)|^2\leq\frac{3}{\lambda_1(1-\textrm{e}^{-\lambda_1T_1})^4}\int_0^{2T_1}
|f_j(s)|^2ds,\nn\en since  the definition of the function $\beta_j$
in (3.5). \ \par  Hence \be\ds \|\tilde u(x,0\|^2=\sum_{j=1}^\infty
|g_j^\infty(0)|^2\leq\frac{3}{\lambda_1(1-\textrm{e}^{-\lambda_1T_1})^4}\sum_{j=1}^\infty\int_0^{2T_1}
|f_j(s)|^2ds=C(T_1)\|f\|^2_{L^2(Q_{2T_1})}\nn\en which proves the
first
part of the theorem. \\
It remains to  prove that the function $\tilde u(x,t)$ has the drift
$T_1$ periodic property for the norm $L^2$. For that purpose
consider $\|\tilde u(t+T_1)-\tilde u(t)\|^2\equiv\Gamma(t).$ So he
have from Lemma 2.3 and theorem 2.4  \be
\ds\Gamma(t)=\sum_{j=1}^\infty
|g_j^\infty(t+T_1)-g_j^\infty(t)|^2=\sum_{j=1}^\infty|\gamma_j(t)|^2
\nn\en where \begin{equation}
\gamma_j(t)=\frac{\textrm{e}^{-\lambda_jt}}{1-\textrm{e}^{-\lambda_jT_1}}\int_0^{T_1}
\beta_j(s)\textrm{e}^{-\lambda_j(T_1-s)}ds+ \int_0^{t}
\beta_j(s)\textrm{e}^{-\lambda_j(t-s)}ds. \tag{3.8}\end{equation}
Using again the inequality of Cauchy-Schwarz we find from (3.8)
that $\Gamma(t)\leq C(T_1)\|f\|^2_{L^2(Q_{2T_1})}$, $C(T_1)$ being a
generic constant depending on $T_1$, which proves that $\Gamma(t)$
is well defined. The property $T_1$-periodic follows
since the functions $\gamma_j(t)$ are themselves $T_1$-periodic. \\
Using Lemma 2.3 yet we give an asymptotic expansion for the solution
of the problem (3.1)-(3.3). So we have \\{\bf Theorem 3.3} \it Let
$u$ be the solution of the initial and boundary value problem
(3.1)-(3.3) then there exists a function $u_1\in
L^2_{Loc}(Q_\infty)$ such as \be \ds\lim _{n\rightarrow
\infty}\|u(t+nT_1)-\tilde u(t)-nu_1(t)\|=0 \nn\en $\tilde u$
denoting the solution given in theorem 3.2.\\ Proof:
 \rm Let $u(x,t)=\sum_{j=1}^\infty g_j(t)w_j(x)$ the unique
solution of the problem (3.1)-(3.3). From Lemma 2.3 the coefficient
$g_j(t)$ have the following expansion \begin{equation}
g_j(t+nT_1)=g_j^\infty(t)+\delta_{n,j}(t)+n
\tilde\gamma_j(t)\tag{3.9}\end{equation} where \\ \be
\tilde\gamma_j(t)=\frac{\textrm{e}^{-\lambda_jT_1}}{1-\textrm{e}^{-\lambda_jT_1}}\int_0^{T_1}
\beta_j(s)\textrm{e}^{-\lambda_j(t-s)}ds+ \int_0^{t}
\beta_j(s)\textrm{e}^{-\lambda_j(t-s)}ds \nn\en \be \delta_{n,j}(t)=
-\textrm{e}^{n\lambda_jT_1}[\textrm{e}^{-jt}u_{0j}&-&\frac{\textrm{e}^{-\lambda_jT_1}}{1-\textrm{e}^{-\lambda_jT_1}}\int_0^{T_1}
f_j(s)\textrm{e}^{-\lambda_j(t-s)}ds
\nn\\&+&\frac{\textrm{e}^{-\lambda_jT_1}}{(1-\textrm{e}^{-\lambda_jT_1})^2}\int_0^{T_1}
\beta_j(s)\textrm{e}^{-\lambda_j(t-s)}ds ]\nn\en It is easily to
prove as before that $\sum_{j=1}^\infty|\tilde\gamma_j(t)|^2\leq
C(T_1\|f\|^2_{L^2(Q_{2T_1})},$ . So the function
$u_1(x,t)=\sum_{j=1}^\infty\tilde\gamma_j(t)w_j(x)$ is well defined.
We have \be \|u(t+nT_1)-\tilde u(t)-nu_1(t)\|^2=\sum_{j=1}^\infty
|\delta_{n,j}(t)|^2\nn\en Using the definition of $\delta_{n,j}(t)$
in (3.9) we get the following bound \be |\delta_{n,j}(t)|^2\leq
3\textrm{e}^{-2\lambda_jt-2n\lambda_jT_1}\left[u^2_{0j}+
\frac{6}{\lambda_1(1-\textrm{e}^{-\lambda_1T_1})^4}\int_0^{T_1}f_j^2(s)ds\right]\nn\en
and the last inequality enables us to obtain \be \sum_{j=1}^\infty
|\delta_{n,j}(t)|^2\leq C\textrm{e}^{-2nT_1}\rightarrow 0
\quad\textrm{as}\quad n \rightarrow \infty, \quad C
\quad\textrm{constant} \nn\en which proves our assertion.\ \par \it
Remark: \rm The three previous theorems are still valid with
$\rho(t) \;T_1$-periodic.

\section{Connection with the semi-group}

1. \bf Preliminaries: \rm Let $X$ be a Banach space with
norm$\|\cdot\|$. We consider the linear evolution equation given
by
\begin{equation}\tag{4.1}
x^{\prime}(t)  = -A(t)x(t) + f(t), \, t \in \mathbb{R}_+,
\end{equation}
where $A(t)$ is a family of closed linear operators in $X$ and
$f(t)$ be an $X$ valued function. Throughout this paper, we make the
following assumptions

\vspace{0.2cm} \noindent {\bf Assumption 1.} For each initial value
$x(0) = \zeta$ in $X$, there exists a unique mild solution $x$ of
equation (4.1) on $\mathbb{R}_+$, defined by
\begin{equation}\tag{4.2}
x(t) = U(t,0)\zeta + \int_{0}^{t}U(t, s)f(s)\,ds,
\end{equation}
where $U(t, s)$, \, $0 \leq s \leq t$, is the evolution system
associated with the family $\left\{A(t)\right\}, \,t \in
\mathbb{R}_+$.

\vspace{0.2cm} \noindent {\bf Assumption 2.} The maps $t \longmapsto
A(t)$ is $\eta$ - periodic and, for each $t \in \mathbb{R}_+$,
$A(t)$ is dissipative operator, that is for every $x\in
D\left(A(t)\right)$, there exists a $x^* \in F(x)$ such that
\begin{equation}\tag{4.3}
Re\left<Ax, x^{*}\right> \leq 0,
\end{equation}
where $$F(x) = \left\{x^{*} \in X^{*}: \left<x, x^*\right> =
\|x\|^{2} = \|x^*\|^2\right\}.$$

\vspace{0.2cm} \noindent {\bf Assumption 3.} There exists a $X$
valued function $\beta(t)$, $t \geq 0$, which is $\eta$ - periodic
such that
\begin{equation}\tag{4.4}
f(t + \eta) = f(t) + \beta(t), \, \forall t \in \mathbb{R}_+.
\end{equation}

\noindent {\bf  \underline{Remark}.}
\begin{enumerate}
\item Assumption 1 follows from the following three conditions (Pazy [15])
\begin{itemize}
\item[(i)] The domain $D(A)$ of $\left \{A(t): t \in \mathbb{R}_+\right\}$ is dense in $X$ and independent of $t$.
\item[(ii)] For each $t \in \mathbb{R}_+$, the resolvent $R\left(\lambda: A(t)\right)$ of $A(t)$ exists for all $\lambda$
with $Re(\lambda) \leq 0$ and there exists a constant $M$ such that
\begin{equation}\tag{4.5}
\left\|R\left(\lambda: A(t)\right)\right\| \leq
\frac{M}{|\lambda|+1}.
\end{equation}
\item[(iii)] There exist a constants $L$ and $0 < \alpha \leq 1$ such that
\begin{equation}\tag{4.6}
\left\|\left(A(t) - A(s)\right)A^{-1}(\tau)\right\| \leq L|t -
s|^{\alpha}, \,\, \forall t, s, \tau \in \mathbb{R}_+.
\end{equation}
\end{itemize}
\item The periodicity of $A(t)$ implies that $U(t + \eta, s + \eta) = U(t, s)$, for all $0 \leq s \leq t$.
\item From the assumptions 2, we deduce that $\|U(t, s)\| \leq 1$ for each $0 \leq s \leq t$.
\end{enumerate}

2. \bf Results:  \rm In this section, we prove that there exists a
unique solution $x_{\infty}(t)$ of equation (4.1)satisfying
\begin{equation}\tag{4.7}
x_{\infty}(t + \eta) = x_{\infty}(t) + \gamma(t), \,\, \text{for
all} \,\, t \geq 0,
\end{equation}
where $\gamma(t)$ is an $\eta$-periodic function. First, we need the
following lemma

\vspace{0.2cm} \noindent {\bf \underline{Lemma 4.1}.} Let $x(t)$ is
a mild solution of equation (4.1) with the initial value $x(0) =
\zeta$. For $n \in \mathbb{Z}_+$ and $t \in [0, +\infty)$, we put
\begin{equation}\tag{4.8}
x_{n}(t) = x(t + n\eta) = U(t + n\eta, 0)\zeta + \int_{0}^{t +
n\eta}U(t+n\eta, s)f(s)\,ds.
\end{equation}
Then we have
\begin{equation}\tag{4.9}
x_{n}(t) = x_{\infty}(t) + \delta_{n}(\eta) + n\gamma(t),
\end{equation}
where \begin{multline}\tag{4.10}
\hspace{1cm} x_{\infty}(t) = U(t, 0)\left\{(I - U_0)^{-1}\int_{0}^{\eta}U(\eta, s)f(s)\,ds \right. \\
\left. - (I - U_0)^{-2}\int_{0}^{\eta}U(\eta,
s)\beta(s)\,ds\right\}+\int_{0}^{t}U(t, s)f(s)\,ds \hspace{1cm}
\end{multline}
\begin{equation}\tag{4.11}
\delta_{n}(t) = U(t, 0)\,U_{0}^{n}\left[\zeta - (I -
U_0)^{-1}\int_{0}^{\eta}U(\eta, s)f(s)\,ds + (I -
U_0)^{-2}\int_{0}^{\eta}U(\eta, s)\beta(s)\,ds\right],
\end{equation}
\begin{equation}\tag{4.12}
\gamma(t) = U(t, 0)(I - U_{0})^{-1}\int_{0}^{\eta}U(\eta,
s)\beta(s)\,ds + \int_{0}^{t}U(t, s)\beta(s)\,ds,
\end{equation}
with $U_{0} = U(\eta, 0)$ such that $\|U_{0}\|<1.$

\begin{proof}
By the assumption 3, the remark (ii), it follows from (4.8) that
\begin{alignat}{2}\tag{4.13}
 x_{n}(t) & = U(t + n\eta, 0)\zeta + \int_{0}^{n\eta}U(t + n\eta, s)f(s)ds + \int_{n\eta}^{t+n\eta}U(t + n\eta, s)f(s)ds \notag\\
  &= U(t + n\eta, n\eta)\left\{U(n\eta, 0)\zeta +\int_{0}^{n\eta}U(n\eta, s)f(s)ds\right\} +
\int_{0}^{t}U(t+ n\eta, s + n\eta)f(s+n\eta)ds \notag\\
 & = U(t, 0)x_{n}(0) + \int_{0}^{t}U(t, s)f(s)ds + n\int_{0}^{t}U(t,
 s)\beta(s)ds \notag
\end{alignat}
since we have $U(t+nT,s)=U(t+nT,nT)U(nT,s).$  On the other hand, we
have
\begin{alignat}{2}\tag{4.14}
x_{n}(0)  &= U(n\eta, 0)\zeta + \int_{0}^{n\eta}U(n\eta, s)f(s)ds \notag \\
 &= U(n\eta, 0)\zeta + \sum_{k=1}^{n}\int_{(k-1)\eta}^{k\eta}U(n\eta, s)f(s)ds \notag\\
&=U(n\eta, 0)\zeta + \sum_{k=1}^{n}\int_{0}^{\eta}U(n\eta, s + (k-1)\eta)f(s)ds \notag\\
& \hspace{2cm} + \sum_{k=1}^{n}(k-1)\int_{0}^{\eta}U(n\eta, s +
(k-1)\eta)\beta(s)ds \notag
\end{alignat}
Using the following relations
\begin{equation}\tag{4.15}
U(n\eta, 0) = U^{n}(\eta, 0) = U_{0}^{n},
\end{equation}
and
\begin{equation}\tag{4.16}
U(n\eta, s+ (k-1)\eta) = U_{0}^{n-k}U(\eta, s),
\end{equation}
we deduce from (4.14) that
\begin{equation}\tag{4.17}
x_{n}(0) = U_{0}^{n}\zeta +
\left(\sum_{k=1}^{n}U_{0}^{n-k}\right)\int_{0}^{\eta}U(\eta,
s)f(s)ds +
\left(\sum_{k=1}^{n}(k-1)U_{0}^{n-k}\right)\int_{0}^{\eta}U(\eta,
s)\beta(s)ds.
\end{equation}
Since
\begin{equation}\tag{4.18}
\sum_{k=1}^{n}U_{0}^{n-k} = (I-U_{0})^{-1}(I-U_{0}^{n}),
\end{equation}
and
\begin{equation}\tag{4.19}
\sum_{k=1}^{n}(k-1)U_{0}^{n-k} = - (I- U_{0})^{-2} + (I-
U_{0})^{-2}U_{0}^{n} + n(I - U_{0})^{-1},
\end{equation}
it follows from (4.15) that
\begin{multline}\tag{4.20}
x_{n}(0) = (I-U_{0})^{-1}\int_{0}^{\eta}U(\eta, s)f(s)ds - (I- U_{0})^{-2}\int_{0}^{\eta}U(\eta, s)\beta(s)ds \\
+ U_{0}^{n}\left\{\zeta - (I-U_{0})^{-1}\int_{0}^{\eta}U(\eta,
s)f(s)ds + (I- U_{0})^{-2}\int_{0}^{\eta}U(\eta, s)
\beta(s)ds\right\}\\
+n(I-U_{0})^{-1}\int_{0}^{\eta}U(\eta, s)\beta(s)ds .
\end{multline}
Combining (4.13) and (4.20) we obtain (4.9). The proof of Lemma is
complete.
\end{proof}

\noindent {\bf \underline{Theorem 4.2}.} Let the assumptions 1, 2, 3
hold. Then, there exists a unique solution $x_{\infty}(t)$ of (4.1)
such that
\begin{equation}\tag{4.21}
x_{\infty}(t+\eta) = x_{\infty}(t) + \gamma(t), \,\, \forall t \geq
0,
\end{equation}
where $\gamma(t)$ is an $\eta$-periodic function, defined by (4.12).

\begin{proof}
For $n \in \mathbb{Z}_+$ and $t \geq 0$, put
\begin{equation}\tag{4.22}
u_n(t) = x_{n}(t) - n\gamma(t).
\end{equation}
By Remark 3, it follows (4.9)-(4.12) and (4.22) that
\begin{equation}\tag{4.23}
\lim_{n\rightarrow \infty}u_{n}(t) = x_{\infty}(t).
\end{equation}
It is clear that $x_{\infty}$ is a mild solution of equation (4.1),
satisfies the initial value
\begin{equation}\tag{4.24}
x_{\infty}(0) = (I- U_0)^{-1}\int_{0}^{\eta}U(\eta, s)f(s)ds - (I-
U_0)^{-2}\int_{0}^{\eta}U(\eta, s)\beta(s)ds.
\end{equation}
From (4.23), we have
\begin{alignat}{2}\tag{4.25}
x_{\infty}(t + \eta) & = \lim_{n \rightarrow \infty}u_{n}(t + \eta)
= \lim_{n\rightarrow \infty}\left\{x_{n}(t + \eta) -
n\gamma(t + \eta)\right\} \notag \\
& = \lim_{n\rightarrow \infty}\left\{x_{n+1}(t) - (n + 1)\gamma(t) +
n\left[\gamma(t) - \gamma(t+\eta)\right] + \gamma(t) \right\}.
\notag
\end{alignat}
On the other hand, by the periodicity of $\beta(t)$, we get
\begin{alignat}{2}\tag{4.26}
\gamma(t + \eta) &= U(t + \eta, 0)(I -
U_{0})^{-1}\int_{0}^{\eta}U(\eta, s)\beta(s)ds + \int_{0}^{t+ \eta}
U(t + \eta, s)\beta(s)ds \notag \\
& = U(t , 0)(I - U_{0})^{-1}\int_{0}^{\eta}U(\eta, s)\beta(s)ds +
\int_{\eta}^{t+ \eta}
U(t + \eta, s)\beta(s)ds \notag \\
& = U(t , 0)(I - U_{0})^{-1}\int_{0}^{\eta}U(\eta, s)\beta(s)ds +
\int_{0}^{t} U(t, s)\beta(s)ds \equiv \gamma(t) \notag
\end{alignat}
Combining (4.25), (4.26) we obtain
\begin{equation}\tag{4.27}
x_{\infty}(t + \eta) = \lim_{n \rightarrow \infty}u_{n + 1}(t) +
\gamma(t) = x_{\infty}(t) + \gamma(t).
\end{equation}
Now, let $\widetilde{x}(t)$ be the mild solution of the equation
(4.1) corresponding to the initial value $\widetilde{x}(0)$ and
\begin{equation}\tag{4.28}
\widetilde{x}(t + \eta) = \widetilde{x}(t) + \widetilde{\gamma}(t),
\end{equation}
where $\widetilde{\gamma}(t)$ is an $\eta$-periodic function. Then
$\widehat{x}(t) = x_{\infty}(t) - \widetilde{x}(t)$ satisfy
\begin{equation}\tag{4.29}
\left\{
\begin{array}{l}
 \widehat{x}^{\prime}(t) = A(t)\widehat{x}(t), \,\, t \geq 0, \\
 \widehat{x}(0) = x_{\infty}(0) - \widetilde{x}(0), \\
\end{array}
\right.
\end{equation}
and
\begin{equation}\tag{4.30}
\widehat{x}(t+\eta) = \widehat{x}(t) + \widehat{\gamma}(t), \quad
\widehat{\gamma}(t + \eta) = \widehat{\gamma}(t), \quad \forall t
\geq 0.
\end{equation}
It follows from (4.29) that
\begin{equation}\tag{4.31}
\widehat{x}(t) = U(t, 0)\left(x_{\infty}(0) -
\widetilde{x}(0)\right).
\end{equation}
From (4.30), (4.31) we deduce that $x_{\infty}(0) =
\widetilde{x}(0)$. By the assumption 1, Theorem 4.2 is proved.
\end{proof}

\section{\textbf{Numerical results }}

$\qquad $In this section we consider the initial and boundary value problem (3.1)-(3.3) of section 3:

\begin{equation}
u_{t}-u_{xx}=f(x,t)\hspace{0.2cm}\text{in}\hspace{0.2cm} (0,1)\times
(0,\infty ),  \tag{5.1}  \label{s1}
\end{equation}

 with Dirichlet boundary conditions
\begin{equation}
u(0,t)=u(1,t)=0 \tag{5.2}  \label{s2}
\end{equation}
 and initial condition
\begin{equation}
u(x,0)=\widetilde{u}_{0}(x)). \tag{5.3}  \label{s3}
\end{equation}
The graphs of the approximated and exact solutions through two examples are very close each other.\ \par
 For the first example the functions $\widetilde {u}_{0},$  and $f$ are defined by
\begin{equation}
\widetilde u_{0}(x)=\frac{4\pi^3}{(\pi^4+4\pi^2)^2}\sin(\pi x)
 \tag{5.4} \label{s4}
\end{equation}

\begin{equation}
f(x,t)= \sin(\pi x)t\sin(2\pi t), \tag{5.5} \label{s5}
\end{equation}
the function $f(x,t)$ satisfying the assumption
$f(x,t+1)=f(x,t)+\beta (x,t)$, $\beta (x,t)$ being 1-periodic in t
for each $x\in [0,1]$. The exact solution of the problem (\ref{s1})
-- (\ref{s3}) with $ \widetilde{u}_{0}$ and $f$ defined in
(\ref{s4}) -- (\ref{s5}) respectively, is the function $u_{ex}$
given by
\begin{equation}
u_{ex}(x,t)=\sin(\pi x)\left[\cos (2\pi t)(a_1t+b_1)+ \sin (2\pi
t)(a_2t-b_2)\right], \tag{5.6} \label{s6}
\end{equation}with $\qquad$
$a_1=\ds\frac{-2\pi}{\pi^4+4\pi^2}$,\;$b_1=\ds\frac{4\pi^3}{(\pi^4+4\pi^2)^2}$,
$a_2=\ds\frac{\pi^2}{pi^4+4\pi^2}$,\;
$b_2=\ds\frac{4\pi^3}{(pi^4+4\pi^2)^2}$ . \
\par

$\qquad $ To solve problem (\ref{s1})-(\ref{s3}) numerically, we
consider the differential system for the unknowns $v_{j}(t)\equiv
u(x_{j},t),$ with $x_{j}=j\Delta x,$ $\Delta x=\frac{1}{p} ,$
$j=0,1,...,p:$ \begin{equation} \left\{ \begin{array}{l}
\ds\frac{dv_1}{dt}(t)=-\frac{2}{\Delta x ^2}v_{1}(t)+\frac{1}{\Delta x ^2}v_{2}(t)+ f_{1}(t) \\
\ds\frac{dv_{j}}{dt}(t)=\frac{1}{\Delta x ^2}v_{j-1}(t) -
\frac{2}{\Delta x ^2} v_j(t)
+\frac{1}{\Delta x ^2}v_{j+1}(t)+f_{j}(t),$\ $j=\overline{2,p-2}\\
\ds\frac{dv_{p-1}}{dt}(t)= \frac{1}{\Delta x
^2}v_{p-2}(t)-\frac{2}{\Delta x ^2}v_{p-1}(t)
+f_{p-1}(t) \\
v_{j}(0)=\widetilde{u}_{0}(x_{j}), f_j(t)=f(x_{j},t),
j=\overline{1,p-1}.
\end{array}%
\right.   \tag{5.7}  \label{s7}
\end{equation} The system (\ref{s7}) is equivalent to:
\begin{equation}
\begin{array}{c}
\frac{d}{dt}X(t)=AX(t)+F(t)\text{,}%
\end{array}
\tag{5.8}  \label{s8}
\end{equation}

\begin{equation}
\left\{
\begin{array}{l}
X(t)=\left( v_{1}(t),...,v_{p-1}(t)\right) ^{T} \\
F(t)=\left( f_{1},...,f_{p-1}\right) ^{T} \\
\end{array}
\right.  \tag{5.9}  \label{s9}
\end{equation} the tridiagonal matrix $A$ being defined by
\begin{equation}
A=\left[
\begin{array}{|lllll|}
\hline
-2\alpha & \alpha  &  &  & 0 \\
\alpha  & -2\alpha  & \alpha  &  &  \\
\multicolumn{1}{|l}{} & \multicolumn{1}{c}{\ddots } &
\multicolumn{1}{c}{
\ddots } & \multicolumn{1}{c}{\ddots } &  \\
\multicolumn{1}{|l}{} & \multicolumn{1}{l}{} &
\multicolumn{1}{r}{\alpha }
& \multicolumn{1}{r}{-2\alpha } & \alpha  \\
\multicolumn{1}{|l}{0} & \multicolumn{1}{l}{} & &
\multicolumn{1}{r}{\alpha } &{-2\alpha }
\\ \hline
\end{array}
\right]  \tag{5.10}  \label{s10}
\end{equation}
where $\alpha=\frac{1}{\Delta x^2}$ \ \par $\qquad $ To solve the
linear differential system (\ref{s9}), we use a spectral method with
a time step $\Delta t=0.05 $ and a spacial step $\Delta x=0.1 $

\begin{center}
\includegraphics[width=10cm,height=10cm,
 keepaspectratio=true]{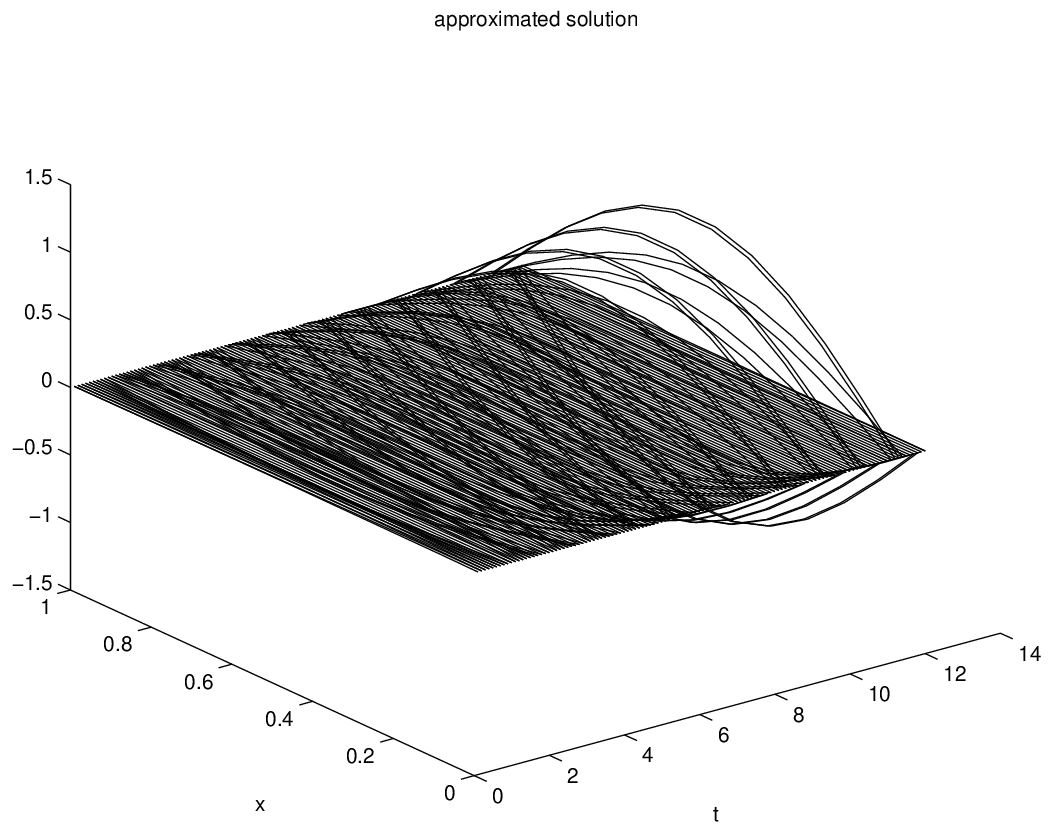}

Figure 2: Numerical solution
\end{center}

In fig.2 we have drawn the approximated solution of the problem
(\ref{s1})-(\ref{s3}) while fig.3 represents his corresponding exact
solution (\ref{s6}).

\begin{center}
\includegraphics[width=10cm,height=10cm,
keepaspectratio=true]{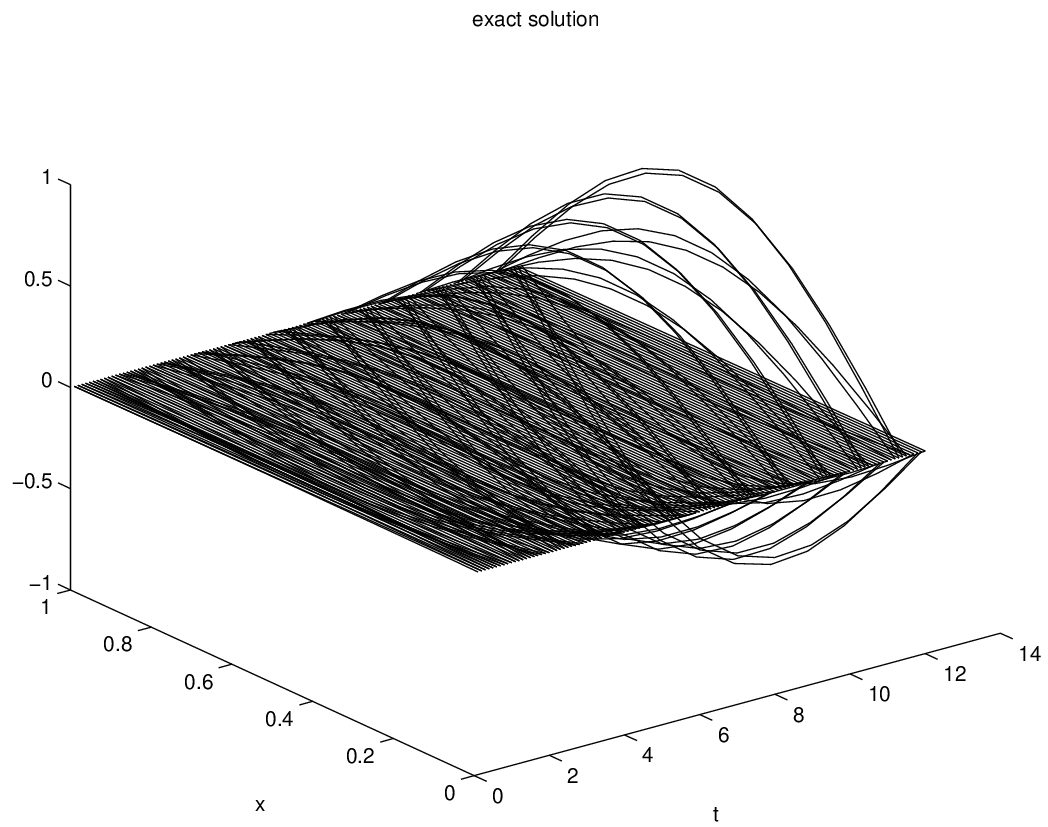}

Figure 3: Exact solution
\end{center}

The fig.4 corresponds to the median approximated component
$(x_j,t)\mapsto u(x_j,t)$. So we can see the drift property of this
component generated by the the
 drift property of the second hand side $f(x,t)$   \\

\begin{center}
\includegraphics[width=7cm,height=7cm,
keepaspectratio=true]{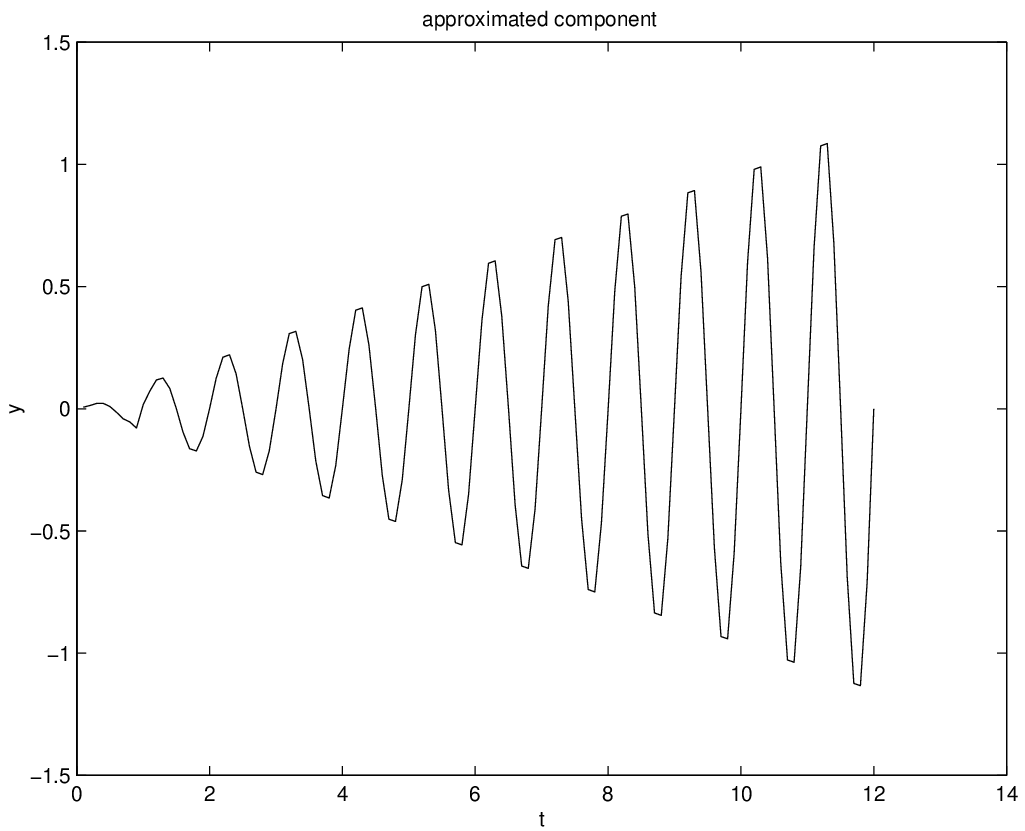}

Figure 4: curve of  $(x_j,t)\mapsto u(x_j,t)$
\end{center}
For one component $x_j$ fixed we find again numerically the curve of the function  $t\mapsto y_n(t)$ given in Lemma 2.3 (fig 1).\ \par
We have also considered another numerical example whose given data
are
\begin{equation}
f(x,t)= t\sin(\pi x), \quad \widetilde{u}_{0}(x)=-\frac{\sin (\pi
x)}{\pi^4} \tag{5.11} \label{s11}
\end{equation}
The exact solution of (\ref{s11}) is $u_{ex}(x,t)=\ds\sin (\pi
x)\left(\frac{t}{\pi^2}-\frac{1}{\pi^4}\right)$. So with the same
method as before, the corresponding surfaces and curve are drawn in
fig.5, fig.6 and fig.7 ( respectively approximated solution, exact
solution and median component $(x_j,t)\mapsto u(x_j,t)$ ).

\begin{center}
\includegraphics[width=10cm,height=10cm,
 keepaspectratio=true]{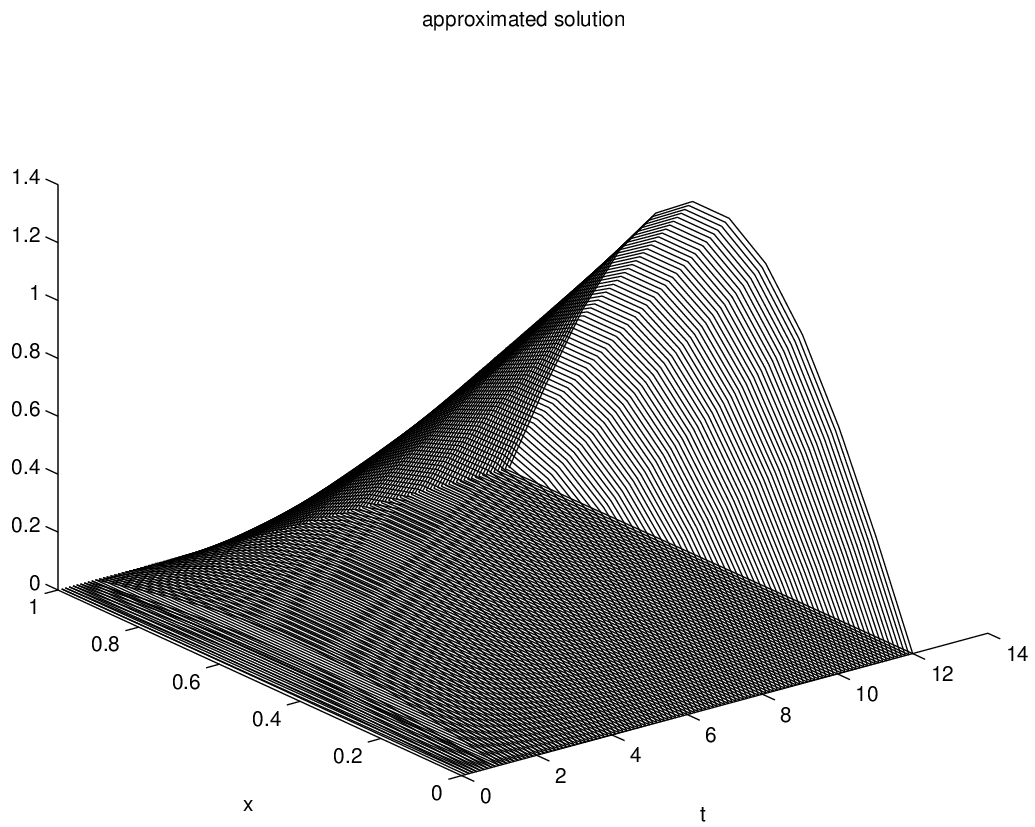}

Figure 5: Numerical solution, case 2
\end{center}

\begin{center}
\includegraphics[width=10cm,height=10cm,
keepaspectratio=true]{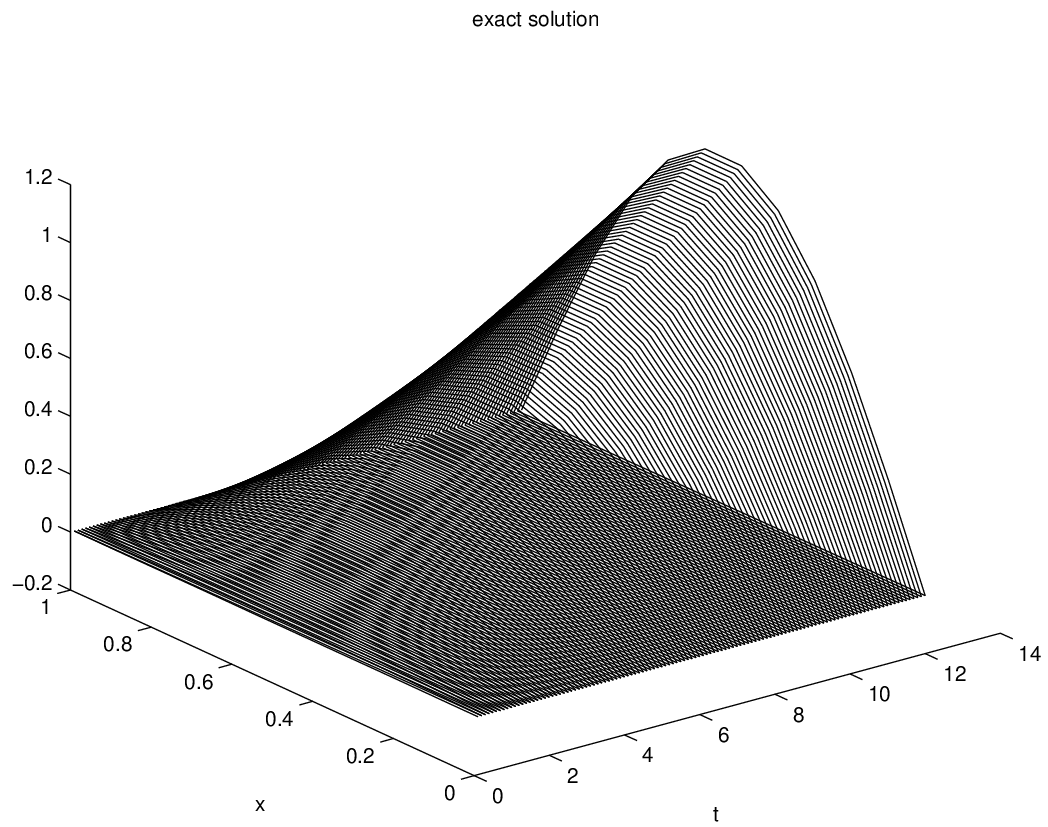}

Figure 6: Exact solution, case 2
\end{center}

\begin{center}
\includegraphics[width=7cm,height=7cm,
keepaspectratio=true]{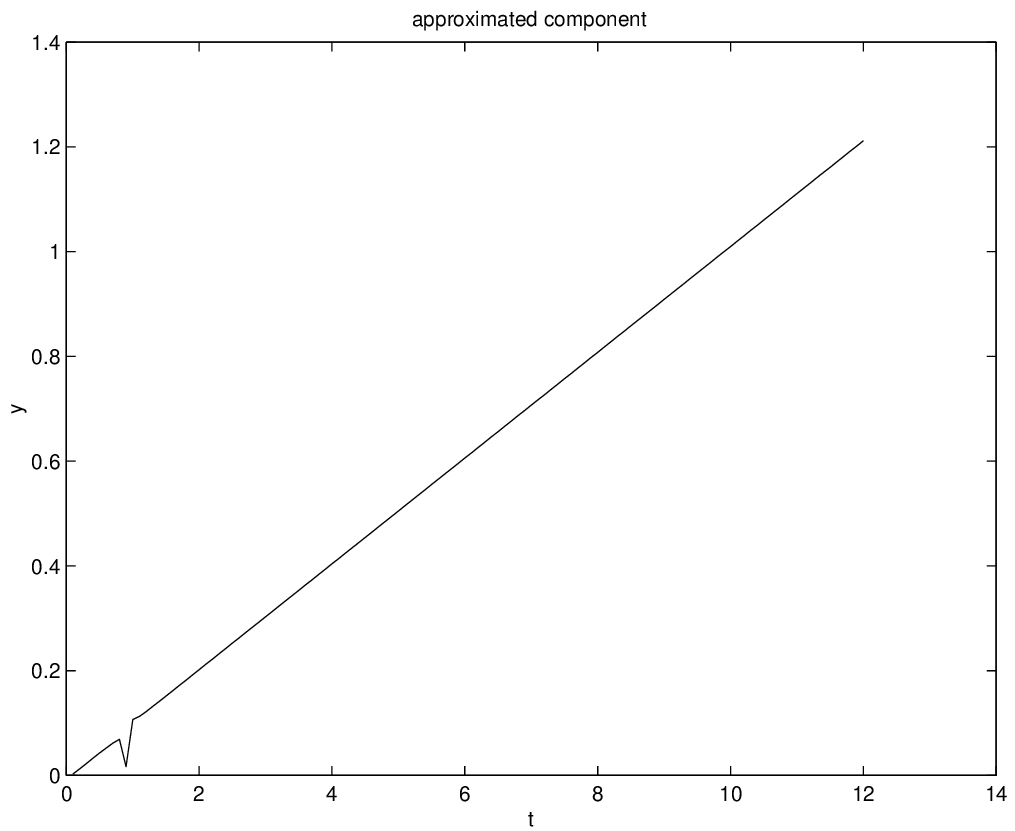}

Figure 7: curve of  $(x_j,t)\mapsto u(x_j,t)$
\end{center}
So through the results of the sections 2 and 3 we can see that a linear drift on the source term (the rhs) gives a linear drift
on the solution. This property can be generalized via an analytic semi-group liking the ODE and PDE. On the other hand there exits a unit concerning the sections 2, 3 and 4 which dwells in the asymptotic behavior (lemma 2.3, Theorem 3.3 and theorem 4.1). \ \par
\vspace{0.3cm}
\textbf{Acknowledgements.} The authors wish to express their sincere thanks
to the referees for the suggestions and valuable comments. The first author (Stephane Cordier) wants to thank M. Martin of INRA for very
helpful discussions. The third author
(Nguyen Thanh Long) is also extremely grateful for the support given by
Vietnam's National Foundation for Science and Technology Development
(NAFOSTED) under Project \textbf{101.01-2010.15.}

\end{document}